\documentclass[12pt]{amsart}

\usepackage{amsmath,amssymb,latexsym} % ams auxiliaries.
\usepackage{fullpage}
\usepackage{etoolbox}
\usepackage{enumitem}
\usepackage{color}
\usepackage[colorlinks,linkcolor=blue,citecolor=magenta]{hyperref}
\usepackage[colorinlistoftodos,bordercolor=orange,backgroundcolor=orange!20,linecolor=orange,textsize=scriptsize]{todonotes}
\usepackage{soul}
\usepackage{microtype}
\usepackage{yhmath}

\newtheorem{theorem}{Theorem} % If subsection, get 3.1.2, etc.
\newtheorem{corollary}{Corollary}
\newtheorem{proposition}{Proposition}

\newtheorem{lemma}{Lemma}

\newtheorem*{proof-claim}{Proof}

\theoremstyle{definition}

\newenvironment{changemargin}[2]{\begin{list}{}{%
\setlength{\topsep}{0pt}%
\setlength{\leftmargin}{0pt}%
\setlength{\rightmargin}{0pt}%
\setlength{\listparindent}{\parindent}%
\setlength{\itemindent}{\parindent}%
\setlength{\parsep}{0pt plus 1pt}%
\addtolength{\leftmargin}{#1}%
\addtolength{\rightmargin}{#2}%
}\item }{\end{list}}

\AtBeginEnvironment{proof-claim}{\vspace{-0.5cm}\footnotesize\begin{changemargin}{0.5cm}{0.5cm}}
\AtEndEnvironment{proof-claim}{\end{changemargin}}

\def\leq{\leqslant}
\def\geq{\geqslant}
\def\A{\mathcal{A}}
\def\NP{\operatorname{NP}}
\def\cd{\operatorname{cd}}

\def\HH{\mathcal{H}}
\def\K{\mathcal{K}}
\def\P{\mathcal{P}}
\def\ST{\operatorname{ST}}
\def\Xind{\operatorname{Xind}}
\def\Hom{\operatorname{Hom}}

\def\N{\mathbb{N}}

\def\stab{\text{\textup{-stab}}}
\def\KG{\operatorname{KG}}
\def\SG{\operatorname{SG}}

\begin{document} 

\title{Colorings of complements of line graphs}

\author{Hamid Reza Daneshpajouh}
\address{H. Daneshpajouh,
Universit\'e Paris Est, CERMICS, 77455 Marne-la-Vall\'ee CEDEX, France}
\email{hr.daneshpajouh@enpc.fr}

\author{Fr\'ed\'eric Meunier}
\address{F. Meunier, Universit\'e Paris Est, CERMICS, 77455 Marne-la-Vall\'ee CEDEX, France}
\email{frederic.meunier@enpc.fr}

\author{Guilhem Mizrahi}
\address{G. Mizrahi, Risk\&{}Co, 92300 Levallois-Perret, France}
\email{guilhem.mizrahi@gmail.com}

\begin{abstract}
Our purpose is to show that complements of line graphs enjoy nice coloring properties. We show that for all graphs in this class the local and usual chromatic numbers are equal. We also prove a sufficient condition for the chromatic number to be equal to a natural upper bound. A consequence of this latter condition is a complete characterization of all induced subgraphs of the Kneser graph $\KG(n,2)$ that have a chromatic number equal to its chromatic number, namely $n-2$. In addition to the upper bound, a lower bound is provided by Dol'nikov's theorem, a classical result of the topological method in graph theory. We prove the $\NP$-hardness of deciding the equality between the chromatic number and any of these bounds.

The topological method is especially suitable for the study of coloring properties of complements of line graphs of hypergraphs. Nevertheless, all proofs in this paper are elementary and we also provide a short discussion on the ability for the topological methods to cover some of our results.
\end{abstract}


\maketitle 

\section*{Introduction}\label{sec:intro}

\subsection*{Context} The study of a hypergraph through coloring properties of the complement of its line graph is a standard approach in topological combinatorics. The complement of such a line graph is in general termed as a ``Kneser graph'' and appears in various works in this area (see the references stated later in the paper). It also plays an important role in other areas of mathematics, e.g., for the study of ideals of $\N$~\cite{alon2009stable} or to get information on the embeddability of simplicial complexes~\cite{Sa90,Sa91}. The classical Kneser graph, whose chromatic number has been determined by L\'aszl\'o Lov\'asz in his breakthrough paper marking the introduction of the Borsuk-Ulam theorem in combinatorics~\cite{Lo78}, is nothing else but the complement of the line graph of a complete uniform hypergraph.

In this paper, we focus on the special case when the hypergraph is actually a graph, i.e., we deal with complements of line graphs of graphs. %Parallel edges are allowed. 
We show that in this case these complements enjoy much stronger properties than those known (or even possible?) for the general case. We avoid any topological argument, keeping everything at a very elementary level. Actually, we give evidence for the inability of the topological method to recover some of our results.

\subsection*{Local chromatic number} Our first contribution deals with the local chromatic number. The {\em local chromatic number} of a graph $G=(V,E)$, denoted by $\psi(G)$ and introduced by Paul Erd\H{o}s et al.~\cite{ErFuHaKoRo86}, is the minimum number $t$ such that there is a proper coloring of $G$ with a number of colors present in any closed neighborhood of a vertex not exceeding $t$. In other words, we have 
$$\psi(G)=\min_c\max_{v\in V}\left|c(N[v])\right|,$$ where the minimum is taken over all proper colorings $c$ of $G$, and where $$N[v]=\{u\in V\colon\text{$u$ is a neighbor of $v$ in $G$}\}\cup\{v\}$$ is the {\em closed neighborhood of $v$}. 

The local chromatic number is at most the usual chromatic number, and they are sometimes equal, e.g., when the graph is perfect; see \cite{KPS05} where it is proved that the local chromatic number is lower bounded by the fractional chromatic number, which is a stronger result. 
With the following theorem, we identify another family of graphs for which the equality between local and usual chromatic numbers holds.

\begin{theorem}\label{thm:loc}
If $G$ is the complement of a line graph, then $\psi(G)=\chi(G)$. 
\end{theorem}

This theorem generalizes a result by G\'abor Simonyi and G\'abor Tardos~\cite[Proposition 4]{SiTa06}, which is the special case when $G$ is the Schrijver graph $\SG(n,2)$, i.e., when the graph of which we take the line graph is the complement of a cycle; see Section~\ref{sec:struct} for a formal definition of Schrijver graphs.

Actually, in any proper coloring of the complement of a line graph, the number of colors in the closed neighborhood of {\em every} vertex is high: it is at least $\chi(G)-1$. 
Section~\ref{sec:loc} provides a proof of this somehow surprising result together with a proof of Theorem~\ref{thm:loc}.

\subsection*{Vertex cover number and colorability defect}
The following inequalities hold for any graph $G$ that is the complement of the line graph of a graph $H$:
\begin{equation}\label{eq:ineq}
\cd_2(H)\leq\chi(G)\leq\tau(H).
\end{equation}
In this relation, $\tau(H)$ is the vertex cover number of $H$, and $\cd_2(H)$ is its {\em $2$-colorability defect}, defined as the minimum number of vertices to remove so that the graph induced by the remaining vertices is bipartite. The upper bound is immediate: color each edge of $H$ with an incident vertex of the vertex cover. The lower bound is a consequence of a famous theorem by Vladimir Dol'nikov~\cite{Do88} we present now. (Actually, in the special case studied here, the inequality can also be proved directly~\cite[Section 6.2]{AHM17}.)

From a hypergraph $\HH=(V,E)$, we build a graph $\KG(\HH)$ as follows: its vertices are the edges of $\HH$; its edges connect the vertices corresponding to disjoint edges in $\HH$. In other words, $\KG(\HH)$ is precisely the complement of the line graph of $\HH$. Do'lnikov's theorem is
\begin{equation}\label{eq:dolni}\chi(\KG(\HH))\geq\cd_2(\HH),
\end{equation}
where $\cd_2(\HH)$ is here more generally defined as the minimum number of vertices to remove so that the hypergraph induced by the remaining vertices is $2$-colorable:
$$\cd_2(\HH)=\min\big\{|X|\colon\left(V\setminus X,\{e\in E\colon e\cap X\}=\varnothing\right)\text{ is $2$-colorable}\big\}.$$ (Being {\em $2$-colorable} for a hypergraph means that we can color its vertices with two colors so that no edge is monochromatic.) 

Our two other main results deal with the case of equalities in \eqref{eq:ineq}. The argument above proving the upper bound actually shows that the chromatic number of $G$ is $\tau(H)$ when $H$ is triangle-free. We give a more general sufficient condition for the equality to hold. %The {\em co-butterfly} graph is the union of a vertex and a cycle of length $4$.
A {\em co-claw} is an isolated vertex together with a triangle. A {\em butterfly} is a union of two triangles sharing exactly one vertex.

\begin{theorem}\label{thm:struct}
Let $H$ be a graph that is not complete, has no induced co-claw, and has no induced butterfly. If $G$ is the complement of the line graph of $H$, then we have $\chi(G)=\tau(H)$.
\end{theorem}

The proof of this theorem, some consequences, and related results are given in Section~\ref{sec:struct}. One of these consequences is a characterization of all vertex-critical induced subgraphs of $\KG(K_n)$ (usually denoted by $\KG(n,2)$).

Hypergraphs $\HH$ for which Dol'nikov's inequality~\eqref{eq:dolni} is actually an equality enjoy nice properties, e.g., equality between circular and usual chromatic numbers of $\KG(\HH)$, as proved by Meysam Alishahi et al.~\cite{AHM17}, extending a breakthrough result by Peng-An Chen~\cite{Ch11}. In the paper by Alishahi et al., the question on the complexity of deciding whether \eqref{eq:dolni} is actually satisfied as an equality was raised. We prove that given a graph $H$ and the complement $G$ of its line graph, deciding $\chi(G)=\cd_2(H)$ is $\NP$-hard. We state the theorem in a more general and a bit weaker form:

\begin{theorem}\label{thm:complex}
Given a hypergraph $\HH$, deciding whether $\chi(\KG(\HH))$ and $\cd_2(\HH)$ are equal is $\NP$-hard.
\end{theorem}

Whether this decision problem is $\NP$-complete, i.e., whether there is a polynomial certificate for the equality, is still open. We also have a similar complexity result for deciding equality between $\chi(G)$ and the upper bound in \eqref{eq:ineq}. These complexity results are proved and discussed in Section~\ref{sec:compl}.

\subsection*{Acknowledgments} We thank Meysam Alishahi and Mat\v{e}j Stehl\'ik for pointing out a few relevant complementary references.

\section{Preliminaries}\label{sec:prel}

\subsection{Notation and elementary properties} Let $H$ be a graph. We denote by $V(H)$ (resp. $E(H)$) the set of its vertices (resp. edges) and by $v(H)$ (resp. $e(H)$) the number of vertices (resp. edges). Given a vertex $x$ of $H$, we denote by $\delta_H(x)$ the set of edges incident to $x$. The independent number of $H$ is denoted by $\alpha(H)$, its clique number by $\omega(H)$, and its vertex cover number by $\tau(H)$. It is well-known that $\alpha(H)+\tau(H)=v(H)$. The $2$-colorability defect, defined in the introduction, is denoted by $\cd_2(H)$. We have the following relation, which is a direct consequence of the definition of the $2$-colorability defect:
\begin{equation}\label{eq:cd+alpha}
v(H)\leq\cd_2(H)+2\alpha(H).
\end{equation}
Denoting by $G$ the complement of the line graph of $H$, these relations involving the independence number of $H$ combined with \eqref{eq:ineq} lead to the following bounds on the chromatic number of $G$ (noted to hold by Paul Renteln~\cite[Theorem 2.1, Item (i)]{renteln2003chromatic} when $\alpha(H)\leq 2$). We state them just for sake of completeness and we will not refer explicitly to them elsewhere in the paper:
$$v(H)-2\alpha(H)\leq\chi(G)\leq v(H)-\alpha(H).$$

\subsection{Parallel edges}\label{subsec:parallel} Let $H$ be a graph, and let $H'$ be a graph obtained from $H$ by replacing an edge $e$ by parallel edges. Denote respectively by $G$ and $G'$ the complements of the line graphs of respectively $H$ and $H'$. Then $\chi(G)=\chi(G')$ since the replacement of $e$ by parallel edges corresponds in $G$ to replicating a vertex without adding edges between the copies. Therefore, even if all proofs are written while assuming that $H$ is a simple graph, many statements of this paper can be straightforwardly extended to the multigraph case as well. For instance, Theorem~\ref{thm:loc} remains true if applied on a multigraph, and in Theorem~\ref{thm:struct}, if given a multigraph, the condition simply needs to be applied on the underlying simple graph.

\subsection{\texorpdfstring{$\ST$}{ST}-partitions} Let $G$ be the complement of the line graph of a graph $H$. Proper colorings of $G$ are exactly partitions of the edge set of $H$ into stars and triangles. A {\em star} is a tree with a vertex, the {\em center} of the star, connected to all other vertices. A single-edge tree is in particular a star. In this particular case, we will always assume that exactly one of the two vertices has been identified as the center, so as to be in a position of always speak of the center of a star without ambiguity. A {\em triangle} is a circuit of length $3$. We call such a partition of $H$ into stars and triangles an {\em $\ST$-partition}. Any element in an $\ST$-partition is called a {\em part}.

We state and prove a lemma that will be useful in the remaining of the paper.

\begin{lemma}\label{lem:min-tri}
Consider an $\ST$-partition $\P$ of a graph $H$ such that $\P$ has a minimum number of triangles. Then we have the following two properties simultaneously:
\begin{enumerate}[label=\textup{(\roman*)}]
\item \label{center} No vertex of a triangle in $\P$ is the center of a star in $\P$.
\item \label{circuit} If $C$ is a circuit of $H$ with no edge belonging to a star in $\P$, then $C$ is a triangle in $\P$.
\end{enumerate}
\end{lemma}

\begin{proof}
We prove~\ref{center}. Suppose that $T$ is a triangle with a vertex $x$ that is also the center of a star $S$. Then we can add the two edges of $T$ incident to $x$ to $S$, and create a new star with the remaining edge of $T$. This leads to a new $\ST$-partition with the same number of parts but fewer triangles. A contradiction.

We prove~\ref{circuit}. Suppose that $C$ is not a triangle in $\P$. For each triangle $T$ having two edges on $C$, we replace those two edges by the third edge of $T$. Then, we replace all triangles having an edge in common with $C$ by stars whose centers are the vertices of $C$. This does not change the number of parts in $\P$, while it decreases the number of triangles. A contradiction.
\end{proof}

Lemma~\ref{lem:min-tri} allows an elementary proof of Dol'nikov's inequality~\eqref{eq:dolni} when $\HH=H$, namely an elementary proof of the inequality $\chi(G)\geq\cd_2(H)$. We provide this proof to illustrate the relevance of this lemma. Consider an $\ST$-partition realizing $\chi(G)$, i.e., an $\ST$-partition with as few parts as possible. We suppose moreover that among all such $\ST$-partitions, we choose one with as few triangles as possible. Remove all centers of stars from $H$. Because of Lemma~\ref{lem:min-tri}, Item~\ref{center}, we get a graph made of edge-disjoint triangles (and possibly isolated vertices). Remove one vertex from each triangle (if possible). Because of Lemma~\ref{lem:min-tri}, Item~\ref{circuit}, we get a forest. In total, the number of removed vertices is at most the number of parts in the $\ST$-partition and we have turned $H$ into a bipartite graph.

\section{Colorful neighborhoods}\label{sec:loc}

\subsection{Another colorful result}

The following result is mentioned in the introduction.

\begin{theorem}\label{thm:color-neigh}
Let $G$ be the complement of a line graph. In every proper coloring of $G$, every vertex has at least $\chi(G)-1$ colors in its closed neighborhood.
\end{theorem}

Theorem~\ref{thm:loc} provides a complementary information: it shows that in any proper coloring of a graph $G$ that is the complement of a line graph, there is actually always a vertex with at least $\chi(G)$ colors in its closed neighborhood.

\begin{proof}[Proof of Theorem~\ref{thm:color-neigh}]
Consider a proper coloring of $G$. Let $H$ be the graph such that $G$ is the complement of its line graph, and let $\P$ be the $\ST$-partition of $H$ corresponding to the proper coloring of $G$.

Choose any vertex $v$ of $G$, and denote by $e=xy$ the corresponding edge in $H$. Denote by $S_x$ the star with $x$ as a center and with $\delta_H(x)$ as an edge set. Denote by $S_y$ the star with $y$ as a center and with $\delta_H(y)\setminus\delta_H(x)$ as an edge set. Add to $\P$ the star $S_x$ and also the star $S_y$ (if this latter has at least one edge). Remove from each part in $\P$ the edges contained in $\delta_H(x)\cup \delta_H(y)$, and keep in $\P$ those that are nonempty after this removal. We get a new $\ST$-partition $\P'$. It corresponds to a new proper coloring of $G$. The vertices whose color has changed are vertices that are not adjacent to $v$. Thus, the number of colors in the closed neighborhood of $v$ has not changed. In this new coloring, the only color that is not in the closed neighborhood of $v$, if it exists, is the one corresponding to the star $S_y$. Therefore, the number of colors in the closed neighborhood of $v$ is at least $\chi(G)-1$, in both the new and the original coloring.
\end{proof}

\subsection{Rainbow paths}

In a properly colored graph, a rainbow path is a path whose vertices have pairwise distinct colors. A conjecture by Saieed Akbari et al.~\cite{akbari2016colorful}, originally proposed in a weaker form by Chiang Lin~\cite{lin2007simple}, states that for every connected graph $G$ with $G\neq C_7$, there exists a proper coloring with $\chi(G)$ colors such that every vertex is the endpoint of a rainbow path with $\chi(G)$ vertices. Alishahi et al.~\cite[Corollary 1]{alishahi2011rainbow} proved the existence of a proper coloring with $\chi(G)$ colors such that every vertex is the endpoint of a rainbow path with $\chi(G)-1$ vertices. Theorem~\ref{thm:color-neigh} has the following immediate corollary, which shows that complements of line graphs enjoy stronger properties regarding the existence of rainbow paths.

\begin{corollary}\label{cor:rainbow}
Let $G$ be the complement of a line graph. In every proper coloring of $G$, every vertex of $G$ is the endpoint of a rainbow path with $\chi(G)-1$ vertices.
\end{corollary}

\subsection{Proof of Theorem~\ref{thm:loc}}

We start with two lemmas. The first one does not need any proof.

\begin{lemma}\label{lem:stars}
Let $\P$ be an $\ST$-partition of a graph $H$. Let $S$ and $S'$ be two stars in $\P$ with the same center. Define $\P'$ as the $\ST$-partition obtained from $\P$ by merging $S$ and $S'$. The number of parts having an edge disjoint from a given edge $e$ does not increase when passing from $\P$ to $\P'$.
\end{lemma}

%\begin{proof}
%If a part is disjoint from $e$ in $\P'$, it correspond to a part disjoint from $e$ in $\P$.
%\end{proof}

\begin{lemma}\label{lem:2star}
Let $\P$ be an $\ST$-partition of a graph $H$. Consider an edge $e=xy$ of $H$ belonging to a star $S$ in $\P$, whose center is $x$. Denote by $\A$ the collection of stars with exactly two edges  and whose vertices distinct from the center are precisely $x$ and $y$. Assume that $\A$ is nonempty.

Define two stars $S_x$ and $S_y$ as follows. The star $S_x$ has $x$ as a center, and its edges are those in $E(S)\cup\left(\bigcup_{S'\in\A}(E(S')\cap\delta_H(x))\right)$. The star $S_y$ has $y$ as a center, and its edges are those in $\bigcup_{S'\in\A}(E(S')\cap\delta_H(y))$. Define then $\P'$ as a new $\ST$-partition obtained by removing from $\P$ the star $S$ as well as all parts in $\A$ and adding to it the stars $S_x$ and $S_y$.

Let $f$ be any edge of $H$. The number of parts having an edge disjoint from $f$ does not increase when passing from $\P$ to $\P'$.
\end{lemma}

\begin{proof}
Suppose first that $f$ belongs to $S_x$. The only part in $\P'$ having possibly an edge $f'$ disjoint from $f$, in addition to those parts already in $\P$, is $S_y$. Since, in $\P$, such an edge $f'$ belongs to a star in $\A$ -- which is a part that is not present in $\P'$ --  the number of parts having an edge disjoint from $f$ does not increase when passing from $\P$ to $\P'$.

Second, suppose that $f$ belongs to $S_y$. The only part in $\P'$ having possibly an edge $f'$ disjoint from $f$, in addition to those parts already in $\P$, is $S_x$. Since, in $\P$, such an edge $f'$ belongs to $S$ or a star in $\A$ -- which are parts that are not present in $\P'$ --  the number of parts having an edge disjoint from $f$ does not increase when passing from $\P$ to $\P'$.

Last, suppose that $f$ belongs neither to $S_x$, nor to $S_y$. If among $S_x$ and $S_y$, only one of them has an edge $f'$ disjoint from $f$, then the number of parts having an edge disjoint from $f$ does not increase when passing from $\P$ to $\P'$: indeed, this edge $f'$ belongs to $S$ or a star in $\A$, which are parts that are not present in $\P'$. The last case to consider is when both $S_x$ and $S_y$ have an edge disjoint from $f$. If these two edges belong to a same part in $\P$, then necessarily this part is a two-edge star in $\A$, and the edge $e$ is then also disjoint from $f$, showing that $S$ has an edge disjoint from $f$ as well. In any case, the number of parts having an edge disjoint from $f$ does not increase when passing from $\P$ to $\P'$.
\end{proof}

\begin{proof}[Proof of Theorem~\ref{thm:loc}]
Let $H$ be the graph such that $G$ is the complement of its line graph. Consider a proper coloring of $G$ realizing $\psi(G)$ (by this, we mean a proper coloring with at most $\psi(G)$ colors in any closed neighborhood of a vertex). Denote by $\P$ the $\ST$-partition of $H$ corresponding to this proper coloring, and by $t$ the number of colors in this proper coloring. We choose the coloring so that $t$ is as small as possible. We deal first with two easy cases.

Consider the case where there are only triangles in $\P$. Pick an edge $e$ in $H$. Apart from the triangle in $\P$ containing $e$, every part has an edge disjoint from $e$. Therefore, the vertex $v$ corresponding to $e$ in $G$ has all $t$ colors in its closed neighborhood. Thus $\psi(G)=t\geq\chi(G)$. Since the local chromatic number is always upper bounded by the usual chromatic number, we get $\psi(G)=\chi(G)$.

Consider the case where $t=\chi(G)$. Then $\psi(G)=\chi(G)$ since in a proper coloring with a minimum number of colors, there is always a vertex with all colors in its closed neighborhood.

The remaining of the proof deals with the case where there is at least one star in $\P$ and $t-1\geq\chi(G)$. Denote by $x$ the center of a star, and let $e=xy$ be any of its edges. Let $\A$ be the collections of all stars with exactly two edges and whose vertices distinct from the center are precisely $x$ and $y$. If $\A$ is of cardinality one, its star is the only part from $\P$ having no edge disjoint from $e$. Hence, $\psi(G)\geq t-1\geq\chi(G)$. So, suppose that $\A$ is of cardinality at least $2$. We cannot apply to $\P$ the transformation described in Lemma~\ref{lem:2star} for the edge $e$: indeed, Lemma~\ref{lem:2star} shows that we would get a new $\ST$-partition $\P'$ corresponding to a proper coloring of $G$ still realizing $\psi(G)$, but with less than $t$ colors; a contradiction with the minimality assumption on $t$. Similarly, if we could decrease the number of colors by merging stars with $x$ as a center (single-edge stars incident to $x$ are considered as having $x$ as their center) or stars with $y$ as a center (single-edge stars incident to $y$ are also considered as having $y$ as their center), Lemma~\ref{lem:stars} shows that we would get a proper coloring of $G$ still realizing $\psi(G)$, but with less than $t$ colors; a contradiction again.

So, apart from the star with $x$ as a center, and the star with $y$ as a center (if it exists), every part in $\P$ has an edge disjoint from $e$. Thus $\psi(G)\geq t-1\geq\chi(G)$. 
\end{proof}

\section{Case of tightness of the vertex cover number}\label{sec:struct}

\subsection{Consequences of Theorem~\ref{thm:struct}}

Let $n$ and $s$ be two positive integers. A subset of $S$ of $[n]$ is {\em $s$-stable} if for any $i$ and $j$ in $S$ we have $s\leq|i-j|\leq n-s$. In other words, $S$ is $s$-stable if it is a clique of the web $W^n_s$. (The {\em web $W^n_s$}, introduced by Leslie Trotter~\cite{trotter1975class}, is a graph with $[n]$ as vertex set and with an edge between $i$ and $j$ precisely when $s\leq|i-j|\leq n-s$.)

Let $k$ be a positive integer such that $n\geq ks$. The {\em $s$-stable Kneser graph} is the graph $\KG(\HH)$ when $\HH$ is the hypergraph with $[n]$ as vertex set and with the $s$-stable $k$-subsets of $[n]$ as edge set. It is denoted by $\KG(n,k)_{s\stab}$. When $s=1$, any $k$-subset is $1$-stable and we get the ``classical'' Kneser graph mentioned in the first paragraph of the introduction. It is also denoted by $\KG(n,k)$, and its chromatic number, established by Lov\'asz~\cite{Lo78}, is $n-2k+2$. Alexander Schrijver~\cite{Sch78} proved that the chromatic number of the $2$-stable Kneser graph $\KG(n,k)_{2\stab}$ -- also called {\em Schrijver graph} and denoted by $\SG(n,2)$ -- is $n-2k+2$ too. The second author~\cite{Me11} conjectured that more generally the chromatic number of $\KG(n,k)_{s\stab}$ is $n-sk+s$ when $s\geq 2$. This conjecture has been proved for even $s$ by Chen~\cite{chen2015multichromatic} and for $s\geq 4$ when $n$ is sufficiently large by Jakob Jonsson~\cite{jonsson2012chromatic}. The case $s=3$ is completely open. (The first author and J\'ozsef Oszt\'enyi~\cite{daneshpajouh2019neighborhood} have recently proved a lower bound off-by-one with respect to the conjectured one.)

The next result is thus in particular the first contribution for the $s=3$ case. It is obtained from Theorem~\ref{thm:struct} with $H=W^n_s$. Independently of this work, Ugo Giocanti has recently obtained an alternate proof of this result. It will be available in his master thesis report.

\begin{corollary}\label{cor:stable}
Let $n$ and $s$ be two positive integers such that $n\geq 2s$ and $s\geq 2$. We have $$\chi(\KG(n,2)_{s\stab})=n-s.$$
\end{corollary}

\begin{proof}
We will apply Theorem~\ref{thm:struct} to $W^n_s$. The graph $W^n_s$ is not complete and has no induced co-claw. Since $\alpha(W^n_s)=s$~\cite{VaVe06}, and thus $\tau(W^n_s)=n-s$, we will get the conclusion once we have proved that $W^n_s$ has no induced butterfly.

Suppose for a contradiction that $W^n_s$ has an induced butterfly. Denote by $a$, $b$, $c$, $d$, and $e$ its five vertices. There are three possibilities:
\begin{center}
\includegraphics[width=10cm]{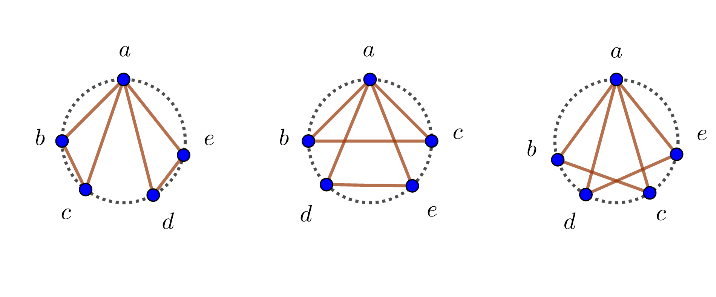}
\end{center}
The vertices are ordered on the circle in the clockwise way. We denote by $\wideparen{xy}$ the arc of the circle between the vertex $x$ and the vertex $y$, in the clockwise way. In each possibility, the arcs $\wideparen{be}$ and $\wideparen{eb}$ contain at least $s$ vertices. Hence, the edge $be$ must be present in $W_s^n$. A contradiction.
\end{proof}

As noted above, $\SG(n,k)$ is an induced subgraph of $\KG(n,k)$ with the same chromatic number. The next result characterizes all induced subgraphs of $\KG(n,2)$ with the same chromatic number, i.e., with chromatic number equal to $n-2$. In particular, it implies that the vertex-critical induced subgraphs of $\KG(n,2)$, with the same chromatic number as $\KG(n,2)$, are exactly those obtained with $\overline H$ being the vertex-disjoint union
\begin{itemize}
\item of cycles of length at least $5$ covering at least $v(H)-1$ vertices, or
\item of cycles of length at least $5$ and a path of length at most $3$, all together covering $V(H)$.
\end{itemize}
Schrijver proved that all $\SG(n,k)$ are vertex-critical. When $k=2$, this is the special case of our result when there is exactly one cycle in the collection. (The question of edge-criticality is probably trickier; see the recent work by Mat\v{e}j Stehl\'ik and Tom\'a\v{s} Kaiser~\cite{kaiser2020edge}, who describe a family of edge-critical spanning subgraphs of $\SG(n,2)$.)
 
The sufficiency condition in the statement below relies on Theorem~\ref{thm:struct}.

\begin{theorem}
Let $G$ be the complement of the line graph of a graph $H$. Then $\chi(G)=v(H)-2$ if and only if $\overline H$ is a collection of vertex-disjoint cycles and paths, in which each cycle is of length at least $5$.
\end{theorem}

\begin{proof}
Suppose first that $\overline H$ is a collection of vertex-disjoint cycles and paths, in which each cycle is of length at least $5$. If $H$ is a complete graph, then we have $\chi(G)=v(H)-2$ (because then $G=\KG(n,2)$ with $n=v(H)$). If $H$ is not the complete graph, then it satisfies the condition of Theorem~\ref{thm:struct}. Since $\tau(H)=v(H)-\alpha(H)=v(H)-\omega(\overline H)=v(H)-2$, we get $\chi(G)=v(H)-2$, as desired.

We now prove the reverse implication. Let $n=v(H)$. The cases $n=3$ and $n=4$ are easily checked. So, suppose that $n\geq 5$. Denote by $H_1$ the graph obtained from $K_n$ by removing the edges of a claw $K_{1,3}$, by $H_2$ the graph obtained from $K_n$ by removing the edges of a triangle, and by $H_3$ the graph obtained from $K_n$ by removing the edges from a cycle of length $4$. Denote by respectively $G_1$, $G_2$, and $G_3$ the complements of the line graphs of $H_1$, $H_2$, and $H_3$. Suppose that $\overline H$ is not a collection of vertex-disjoint cycles and paths as in the statement. The graph $H$ is then a (non-necessarily induced) subgraph of $H_1$, $H_2$, or $H_3$. Hence, $\chi(G)\leq\max(\chi(G_1),\chi(G_2),\chi(G_3))$. The edges of $H_1$ can be covered by one triangle and $v(H)-4$ stars. The edges of $H_2$ can be covered by $v(H)-3$ stars. The edges of $H_3$ can be covered by two triangles and $v(H)-5$ stars. (The two triangles share a vertex, and each of them covers a distinct diagonal of the $4$-cycle.) We get $\chi(G_i)\leq v(H)-3$, and therefore $\chi(G)\leq v(H)-3$.
\end{proof}

\subsection{Proof of Theorem~\ref{thm:struct}}

\begin{proof}[Proof of Theorem~\ref{thm:struct}]
Let $\P$ be an optimal $\ST$-partition with a minimal number of triangles, i.e., we choose an $\ST$-partition of $H$ such that $|\P|=\chi(G)$, and among all $\ST$-partitions satisfying this equality, we choose one with as few triangles as possible. 

We first show that all triangles in $\P$ are vertex-disjoint. Suppose that there are two triangles $T$ and $T'$ sharing a vertex. Since we assume that $H$ is simple (see Section~\ref{subsec:parallel}), it means that $T$ and $T'$ have exactly one vertex in common. Call this vertex $x$. There is no edge between $V(T)\setminus\{x\}$ and $V(T')\setminus\{x\}$ because of Lemma~\ref{lem:min-tri}. But then $T$ and $T'$ form together a butterfly in $H$. A contradiction.

We show now that there is at most one triangle in $\P$. Suppose for a contradiction that there are two distinct triangles $T$ and $T'$. They are vertex-disjoint according to what has just been proved. There cannot be joined by an edge since this edge can neither belong to a triangle in $\P$ (again, this is just what has just been proved), nor belong to a star according to Lemma~\ref{lem:min-tri}. Therefore, any vertex from $T$ forms with $T'$ a co-claw in $H$. A contradiction.

We finish the proof by considering the two possible situations: either there is exactly one triangle in $\P$, or there is no such triangle. Suppose first that there is exactly one triangle $T$ in $\P$. Since $H$ has no induced co-claw, every vertex not in $T$ is incident to an edge whose other endpoint is in $T$. According to Lemma~\ref{lem:min-tri}, every vertex not in $T$ is thus the center of a star in $\P$. Hence, $|\P|=v(H)-2$. This latter quantity is at least $\tau(H)$ since $H$ is not a complete graph. With the upper bound in \eqref{eq:ineq}, we get $\chi(G)=\tau(H)$.

Suppose now that there is no triangle in $\P$. Then the centers of stars in $\P$ form a vertex cover. This implies that $|\P|\geq \tau(H)$. With the upper bound in \eqref{eq:ineq}, we get $\chi(G)=\tau(H)$, again.
\end{proof}

\subsection{Strength of the topological method}
We explain now why we think that neither Corollary~\ref{cor:stable}, nor Theorem~\ref{thm:struct} can be achievable by a direct use of the standard toolbox of the topological method.

There is a lower bound on the chromatic number that dominates all known topological bounds, namely the ``cross-index of the hom-complex'' plus $2$; see~\cite{SiTaZs13} for the exact definition. According to Corollary~\ref{cor:stable}, we have $\chi(\KG(2s+1,2)_{s\stab})=s+1$. Yet, as stated by Proposition~\ref{prop:spec-stab} below, the cross-index bound is $s$ when $s$ is odd. This shows that, to solve the conjecture on the chromatic number of $\KG(n,k)_{s\stab}$, more than a computation of the standard topological bounds is in order. Chen's theorem, which forms the best result obtained so far on this conjecture, was also proved via the topological method, with a clever use of the Tucker lemma. It remains to understand how to systematize his approach and how the bound he obtained compares with the cross-index bound. Anyway, his method was used to deal with the case $s$ even and it is not clear if it can be extended to the case $s$ odd.
 
 \begin{proposition}\label{prop:spec-stab}
 Let $s$ be an odd integer larger than $2$. If $G=\KG(2s+1,2)_{s\stab}$, then we have $$\Xind(\Hom(K_2,G))+2=s.$$
 \end{proposition}

\begin{proof}
For every graph $G$, we have $\Xind(\Hom(K_2,G))+2\geq\omega(G)$. An upper bound on $\Xind(\Hom(K_2,G))+2$ is obtained as follows: if $G$ is such that $\Xind(\Hom(K_2,G))+2\geq t$, then in any proper coloring of $G$ with colors $1,2,\ldots,c$, there is a complete bipartite subgraph $K_{\lfloor t/2\rfloor, \lceil t/2\rceil}$ with the colors of even value on one side and the colors of odd value on the other side. Call such a bipartite graph a {\em zig-zag bipartite graph}. Hence, if we are able to describe a proper coloring of $G$ with no zig-zag bipartite graph $K_{\lfloor t/2\rfloor, \lceil t/2\rceil}$, then we know that $\Xind(\Hom(K_2,G))+2\leq t-1$~\cite[Lemma 3]{SiTaZs13}.

We apply this way of bounding for the special case $G=\KG(2s+1,2)_{s\stab}$. Clearly, we have then $\omega(G)=s$. For the upper bound, color each $2$-stable subset $S\in[2s+1]$ with its minimal element. We get a coloring of $\KG(2s+1,2)_{s\stab}$. Note that the colors are the integers in $[s+1]$, with the colors in $[s]$ being used twice, and the color $s+1$ used only once. The graph $G=\KG(2s+1,2)_{s\stab}$ is the complement of the line graph of a $(2s+1)$-cycle. The coloring of $G$ induces an edge-coloring of this cycle: $1,1,2,2,\ldots,s,s, s+1$, say in the clockwise order. Suppose for a contradiction that there is a zig-zag bipartite graph $K_{\lfloor (s+1)/2\rfloor, \lceil (s+1)/2\rceil}$ in $G$. Its vertices correspond to positions in the sequence of colors above. Since the vertex colored $s+1$ is linked to a vertex of color $1$, the vertices of the zig-zag bipartite graph colored by an odd number occupy the even positions; since the vertex colored $s+1$ is linked to a vertex of color $s$ (here is the place where we use the oddness assumption), the vertices of the zig-zag bipartite graph colored by an odd number occupy the odd positions. A contradiction.
\end{proof}

The gap between the chromatic number of a graph covered by Theorem~\ref{thm:struct} and the cross-index bound can even be arbitrarily large, as we show now. Let $L_m$ be the disjoint union of $m$ cycles of length $7$ and let $F_m$ be the complement of its line graph. Since $L_m$ is triangle-free, we have $\chi(F_m)=\tau(L_m)=4m$. This equality is covered by Theorem~\ref{thm:struct}.  However, we have the following property.

\begin{proposition}
For every integer $m\geq 1$, we have
$$\Xind(\Hom(K_2,F_m))+2=3m.$$
\end{proposition}

\begin{proof}
To bound from below and from above the cross-index bound, we use the same technique as in the proof of Proposition~\ref{prop:spec-stab}. On the one hand, we have $\Xind(\Hom(K_2,F_m)+2\geq\omega(F_m)$ and clearly, $\omega(F_m)=3m$ (the largest clique is obtained by taking a maximum matching from each cycle in $L_m$). On the other hand, color the edges of $L_m$ so that the $i$th cycle is colored like this:
\begin{center}
\includegraphics[width=8cm]{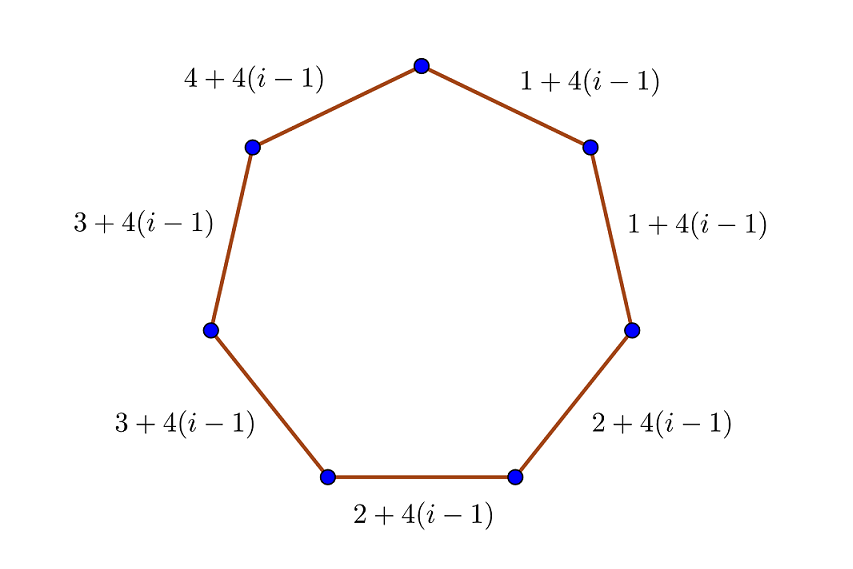}
\end{center}
This induces a proper coloring of $F_m$. The largest zig-zag bipartite graph selects at most $3$ edges in each cycle. Therefore, there is no zig-zag bipartite graph $K_{\lfloor (3m+1)/2\rfloor, \lceil (3m+1)/2\rceil}$ for this coloring in $F_m$, which implies that $\Xind(\Hom(K_2,F_m))+2\leq 3m$.
\end{proof}

\section{Complexity results}\label{sec:compl}

\subsection{Preliminary results}\label{subsec:prel}

Given two graphs $H_1,H_2$ with disjoint vertex sets, the {\em join} of $H_1,H_2$, denoted by $H_1*H_2$, is the graph obtained by taking the union of $H_1$ and $H_2$ and by connecting every vertex of $H_1$ with every vertex of $H_2$. The following lemmas, which will be useful in the proofs of our complexity results, show that the $2$-colorability defect and the vertex cover number behave nicely with the join operation.

\begin{lemma}\label{lem:join-cd}
We always have
$$\cd_2(H_1*H_2)=\min\big(\cd_2(H_1)+v(H_2),\tau(H_1)+\tau(H_2),v(H_1)+\cd_2(H_2)\big).$$
\end{lemma}

\begin{proof}
The left-hand side is at most the right-hand side: removing vertices from $H_1$ (resp. $H_2$) so that it becomes bipartite and removing all vertices from $H_2$ (resp. $H_1$) leads to a bipartite graph; removing the vertices from a vertex cover of $H_1$ and a vertex cover from $H_2$ also lead to a bipartite graph.

Let us prove the reverse inequality. Let $X$ be a subset of $V(H_1*H_2)$ of minimal size such that the subgraph of $H_1*H_2$ induced by $V(H_1*H_2)\setminus X$ is bipartite. We have $|X|=\cd_2(H_1*H_2)$. If $X$ contains $V(H_1)$ (resp. $V(H_2)$), then clearly, we have $|X|\geq v(H_1)+\cd(H_2)$ (resp. $|X|\geq \cd_2(H_1)+v(H_2)$). If $X$ contains none of $V(H_1)$ and $V(H_2)$, then $X$ has to be a vertex cover of both $H_1$ and $H_2$, since otherwise the removal of $X$ would leave at least one triangle. In such a case, we clearly have $|X|\geq \tau(H_1)+\tau(H_2)$.
\end{proof}

\begin{lemma}\label{lem:join-t}
We always have
$$\tau(H_1*H_2)=\min\big(\tau(H_1)+v(H_2),v(H_1)+\tau(H_2)\big).$$
\end{lemma}

\begin{proof}
The left-hand side is at most the right-hand side: removing a vertex cover from $H_1$ (resp. $H_2$) and all vertices from $H_2$ (resp. $H_1$) destroys all edges.

Let us prove the reverse inequality. Let $X$ be a vertex cover of $H_1*H_2$. The set $X$ contains fully $V(H_1)$ or $V(H_2)$, otherwise it would miss an edge. It must also contain a vertex cover of $H_1$ and a vertex cover of $H_2$.
\end{proof}

\subsection{Statements and proofs}

The proof of our complexity results proceeds by a reduction from the problem of determining the independence number of a graph. Given a graph $H$ and an integer $k\geq 1$, we denote by $H(k)$ the graph obtained by taking the join of $H$ with the vertex-disjoint union of $k$ triangles $T_1,\ldots,T_k$. The combination of two propositions will lead directly to the results.

\begin{proposition}\label{prop:Hk-tcd}
The following relations hold:
\begin{eqnarray*}
\cd_2(H(k)) &= &\min\{\cd_2(H)+3k,\tau(H)+2k,v(H)+k\}. \\
\tau(H(k)) &= &\min\{\tau(H)+3k,v(H)+2k\}.
\end{eqnarray*}
\end{proposition}

\begin{proof}
It is a direct consequence of Lemmas~\ref{lem:join-cd} and~\ref{lem:join-t}.
\end{proof}

\begin{proposition}\label{prop:Hk-chr}
Let $G(k)$ be the complement of the line graph of $H(k)$. If $H$ is connected and triangle-free, then we have
$$\chi(G(k))=\min\{\tau(H)+3k,v(H)+k\}.$$
\end{proposition}

\begin{proof}
The inequality $\chi(G(k))\leq\min\{\tau(H)+3k,v(H)+k\}$ is immediate: there are $\ST$-partitions with a number of parts equal to the first and second terms in the minimum. The remaining of the proof is devoted to the reverse inequality.

Consider an $\ST$-partition $\P$ of $H$ such that $\P$ has a minimum number of triangles. Denote by $c$ the number of centers of stars among the vertices of the $T_i$'s. We consider three cases: $c=3k$, $c=3k-1$, and $c\leq 3k-2$.

Suppose first that $c=3k$. According to Lemma~\ref{lem:min-tri}, Item~\ref{center}, there is no triangle in $\P$ sharing a vertex with one of the $T_i$'s. Since $H$ is triangle-free, it implies that every edge of $H$ is contained in a star. The number of these stars must then be at least $\tau(H)$. We get $|\P|\geq \tau(H)+3k$.

Suppose now that $c=3k-1$. If $H$ is reduced to a single edge, then this edge belongs to a part of $\P$, which, together with the $3k-1$ stars centered on the vertices of the $T_i$'s, makes $3k$ parts in total; since $k\geq 1$, we have $3k\geq v(H)+k$ and we get the desired inequality. So, for the remaining of the proof of this case, we assume that $H$ has more than one edge. Denote by $x$ the only vertex of a $T_i$ that is not the center of a star and by $A$ those vertices of $H$ that are centers of stars. If $A=V(H)$, then $|\P|\geq v(H)+3k-1\geq v(H)+k$, and we are done. So, assume that $A\neq V(H)$. By definition of $A$, each edge of $H$ with its two endpoints in $V(H)\setminus A$ belongs to a triangle, and, according to Lemma~\ref{lem:min-tri}, Item~\ref{center}, this triangle must have $x$ has a third vertex. Therefore, the graph induced by $V(H)\setminus A$ is a matching $M$ (the triangles in $\P$ are edge-disjoint and $H$ is simple) and we have $|\P|\geq |A|+|M|+3k-1$. We finish the proof of this case by showing that $|A|+|M|\geq\tau(H)+1$, which implies the desired inequality.

Note that $A$ is not empty, for $H$ is connected and has more than one edge. Pick a vertex $z$ in $A$. The graph $H$ being triangle-free, $z$ has at most one neighbor on each edge of $M$. The neighbors of $z$ in $V(H)\setminus A$, together with one vertex from each edge in $M$ not incident to one of these neighbors, together with $A\setminus\{z\}$ form a vertex cover of $H$ with $|A|-1+|M|$ vertices, which implies that $|A|+|M|\geq\tau(H)+1$. 

Suppose finally that $c\leq 3k-2$. Denote by $x$ and $y$ two vertices of the $T_i$'s that are not centers of a star and by $A$ those vertices of $H$ that are centers of stars. Consider the set $F$ of all edges with one endpoint in $V(H)\setminus A$ and the other in $\{x,y\}$. The set $F$ must be covered by parts in $\P$. Since none of these parts can be a star, they are all triangles, and each of them covers at most two edges in $F$. Thus, we need at least $v(H)-|A|$ parts to cover $F$. Lemma~\ref{lem:min-tri}, Item~\ref{circuit}, implies that no $T_i$ can be covered only by triangles in $\P$, except when $T_i$ itself is in $\P$. This means that we need yet at least another part for each triangle $T_i$. All together, we have $|\P|\geq |A|+v(H)-|A|+k=v(H)+k$.
\end{proof}

\begin{theorem}\label{thm:Hk}
If $H$ is connected and triangle-free, then we have
$$\chi(G(k)) = \cd_2(H(k))\qquad \Longleftrightarrow\qquad \alpha(H) \leq k$$
and
$$\chi(G(k)) = \tau(H(k))\qquad\Longleftrightarrow\qquad \alpha(H) \geq 2k.$$

\end{theorem}

\begin{proof}
Let us prove the first equivalence. 

Suppose that $\alpha(H)$ is at most $k$. It implies that $\tau(H)+2k\geq v(H)+k$ and, with the help of~\eqref{eq:cd+alpha}, that $v(H)\leq\cd_2(H)+2k$. Proposition~\ref{prop:Hk-tcd} shows then that $\cd_2(H(k))=v(H)+k$ and Proposition~\ref{prop:Hk-chr} that $\chi(G(k))=v(H)+k$.

Conversely, suppose that $\chi(G(k))=\cd_2(H(k))$. Since $\cd_2(G(k))$ is at most $\tau(H)+2k$ (using Proposition~\ref{prop:Hk-tcd}), $\chi(G(k))$ is necessarily equal to $v(H)+k$ (using Proposition~\ref{prop:Hk-chr}). This implies in turn that $v(H)+k\leq \tau(H)+2k$ (again, using Proposition~\ref{prop:Hk-tcd}). This latter inequality is equivalent to $\alpha(H)\leq k$.

Let us prove the second equivalence. Suppose that $\alpha(H)$ is at least $2k$. It implies that $\tau(H)+3k\leq v(H)+k$. Proposition~\ref{prop:Hk-tcd} shows then that $\tau(H(k))=\tau(H)+3k$ and Proposition~\ref{prop:Hk-chr} that $\chi(G(k))=\tau(H)+3k$.

Conversely, suppose that $\chi(G(k))=\tau(H(k))$. Since $\chi(G(k))$ is at most $v(H)+k$ (using Proposition~\ref{prop:Hk-chr}), $\tau(H(k))$ is necessarily equal to $\tau(H)+3k$ (using Proposition~\ref{prop:Hk-tcd}). This implies in turn that $\chi(G(k))=\tau(H)+3k$ and thus that $\tau(H)+3k\leq v(H)+k$. This latter inequality is equivalent to $\alpha(H)\geq 2k$.
\end{proof}

Theorem~\ref{thm:complex} is a corollary of this theorem since computing the independence number of a triangle-free graph is $\NP$-hard~\cite{Po74}. Another corollary is the following one.

\begin{corollary}
Let $G$ be the complement of the line graph of a graph $H$. Deciding whether $\chi(G)=\tau(H)$ is an $\NP$-hard problem.
\end{corollary}

\begin{proof}
Consider a connected-triangle free graph $H'$ with at least one edge, pick an arbitrary vertex $x$ of $H'$, and let $H'_x$ be the graph obtained by taking two copies of $H'$ linked by an edge incident to the two copies of $x$. We have $\alpha(H')=\frac 1 2\max_{x\in V(H)}\alpha(H'_x)$. If there were a polynomial algorithm for deciding $\chi(G)=\tau(H)$ for any graph $H$, Theorem~\ref{thm:Hk} shows that applying this algorithm with $H$ equal to $H'_x(k)$ for various values of $k$, we would be able to determine the largest even number bounding from below $\alpha(H'_x)$. Doing this for all possible $H'_x$ (there are $v(H')$ of them) allows to compute in polynomial time $\alpha(H')$. This latter problem is an $\NP$-hard problem, as it has already been noted.
\end{proof}

\subsection{Complementary complexity results}

We gather in this section easy complexity results that might be of some interest. The first remark is that the computation of any of the two bounds in the inequalities~\eqref{eq:ineq} is $\NP$-hard; this is well-known for the vertex cover; for the $2$-colorability defect, see~\cite{AHM17}.

It has been noted in the previous section that computing the independence number of a triangle-free graph is $\NP$-hard~\cite{Po74}. Thus, computing the vertex cover of such a graph is also $\NP$-hard. It implies that computing the chromatic number of the complement of a line graph is $\NP$-hard. We have however the following ``positive'' result, whose proof is very close to the classical proof of the polynomial $2$-approximability of the vertex cover number.

\begin{proposition}\label{prop:2app}
There is a polynomial (and greedy) $2$-approximation algorithm for the chromatic number of complements of line graphs.
\end{proposition}

\begin{proof}
Let $H$ be a graph and let $G$ be the complement of its line graph. Compute any inclusionwise maximal matching of $H$. Denote by $k$ its cardinality. It forms a clique of $G$ and thus $k\leq\chi(G)$. On the other hand, the edges of $H$ can be partitioned into stars centered at the endpoints of the edges of the matching. Such a  partition into stars forms a proper coloring of $G$ with at most $2k\leq 2\chi(G)$ colors.
\end{proof}

Proposition~\ref{prop:2app} is stated for the chromatic number but the proof makes clear that a proper coloring satisfying this approximation ratio can also be computed in polynomial time.

Another easy complexity result is the following one, which is a direct corollary of Theorem~\ref{thm:loc}. This result has originally been proved by Nebojs\v{a} Gvozdenovi\'{c} and Monique Laurent~\cite{gvozdenovic2008operator}, who actually proved a much stronger result holding for many graph parameters.́ Our approach forms thus an alternate proof.

\begin{proposition}
Computing the local chromatic number of a graph is $\NP$-hard.
\end{proposition}

\section{Remarks}

\subsection{When \texorpdfstring{$\overline{H}$}{bar H} admits a planar embedding}

Lemma~\ref{lem:min-tri} can also be used to provide a more elementary proof of the following result by Renteln~\cite[Theorem 2.1, Item (ii)]{renteln2003chromatic} than the original one. This latter used a theorem by Karanbir Sarkaria~\cite{Sa91}, which itself relies on the Borsuk-Ulam theorem. Here, the topology is limited to the study of the possible planar embeddings of few small graphs.

A triangle ``bounds'' a face if there is a bounded face -- a face that is not the outer face -- whose boundary is exactly that triangle. (The {\em boundary} of a face is the subgraph formed by the vertices and the edges incident to the face.)

\begin{proposition}
If $\overline{H}$ admits a planar embedding in which each triangle bounds a face, then $\chi(G)\geq v(H)-3$.
\end{proposition}

\begin{proof}
Let $\P$ be an optimal $\ST$-partition with a minimal number of triangles, i.e., we choose an $\ST$-partition of $H$ such that $|\P|=\chi(G)$, and among all $\ST$-partitions satisfying this equality, we choose one with as few triangles as possible. Because $\overline{H}$ is planar, and thus does not contain a copy of $K_{3,3}$, any two triangles in $\P$ either share a vertex, or are linked by an edge. Suppose for a contradiction that they are linked by an edge. Lemma~\ref{lem:min-tri}, Item~\ref{center}, implies that this edge belongs to another triangle in $\P$. The complement of these three triangles has $K_5$ as a minor. A contradiction. Therefore, any two triangles in $\P$ share a vertex. 

Because of Lemma~\ref{lem:min-tri}, Item~\ref{circuit}, it implies that all triangles in $\P$ actually share a same vertex. The condition in the statement implies in particular that $\alpha(H)$ is at most $3$: in any planar embedding of the graph $K_4$, there is a triangle that is the boundary only of the outer face. Thus the number of triangles in $\P$ is at most $3$. Actually, the complement of three triangles sharing a vertex is the $1$-skeleton of a triangulation of $\mathcal{S}^2$ plus an isolated vertex. In any planar embedding of such a graph, with the condition that each triangle bounds a face, there is again a triangle that is the boundary only of the outer face. Therefore, the number of triangles in $\P$ is zero, one, or two.

Suppose that  there is no triangle in $\P$. Then the centers of the stars in $\P$ form a vertex cover of $H$, which is of cardinality at least $v(H)-\alpha(H)\geq v(H)-3$. 

Suppose that there is exactly one triangle in $\P$. Consider the set $A$ of vertices that are not linked by an edge to this triangle. We claim that they form a clique in $H$: any two such vertices not linked by an edge in $H$ would form with the triangle in $\P$ the following graph in $\overline{H}$: the join (see the definition in Section~\ref{subsec:prel}) of an edge with an independent set of size $3$, and, in any planar embedding of such a graph, there is a triangle that is the boundary of no face, except maybe the outer one. To cover the edges in the clique induced by $A$, we need at least $|A|-1$ stars. Because of Lemma~\ref{lem:min-tri}, Item~\ref{center}, all vertices not in $A$ and not in the triangle, are centers of stars. Therefore, in total, we have $|\P|\geq v(H)-|A|-3+1+|A|-1=v(H)-3$.

Suppose that there are exactly two triangles in $\P$. If there were a vertex not linked by an edge to any of these two triangles, then we would have in $\overline{H}$ the join of a vertex with the complement of a butterfly, and in any planar embedding of such a graph, there is a triangle that is the boundary of no face, except maybe the outer one. Therefore, all vertices not in the triangles are centers of stars and we have in total $|\P|\geq v(H)-5+2=v(H)-3$.
\end{proof}

\subsection{New examples of hypergraphs \texorpdfstring{$\HH$}{H} satisfying  \texorpdfstring{$\chi(\KG(\HH))=\cd_2(\HH)$}{}}

We can extend the definition of the join given in Section~\ref{subsec:prel} to hypergraphs as follows. Given two hypergraphs $\HH_1,\HH_2$ with disjoint vertex sets, the {\em join} of $\HH_1,\HH_2$, denoted by $\HH_1*\HH_2$, is the hypergraph obtained by taking the union of $\HH_1$ and $\HH_2$ and by connecting every vertex of $\HH_1$ with every vertex of $\HH_2$. So, in addition to the original edges of $\HH_1$ and $\HH_2$, their join contains also all possible edges of size $2$ with a vertex from each of them. 

The main point of this section is that Lemmas~\ref{lem:join-cd} and~\ref{lem:join-t} remain true for hypergraphs, with this extended definition of join, and that the construction of Section~\ref{sec:compl} allows to build hypergraphs for which Dol'nikov's lower bound is tight, i.e, for which the chromatic number of their Kneser graph is equal to their $2$-colorability defect. As noted in the introduction, such hypergraphs enjoy nice properties. Yet, just a few examples of such hypergraphs are known: they comprise complete uniform hypergraphs (this is actually Lov\'asz's theorem) or joins of complete uniform hypergraphs under some conditions~\cite[Proposition 2]{AHM17}. As in Section~\ref{sec:compl}, the construction we propose here involve the join operation.

The $2$-colorability defect $\cd_2(\HH)$ of a hypergraph $\HH$ is defined in the introduction. Similarly as for graphs, $v(\HH)$ is the number of vertices of $\HH$, and $\tau(\HH)$ is its vertex cover number, i.e., the minimal cardinality of a subset of vertices intersecting every edge. Here is the hypergraph counterpart of Lemmas~\ref{lem:join-cd} and~\ref{lem:join-t}, stated without proofs, for the proofs of their original versions for graphs can be adapted in a straightforward manner.

\begin{lemma}\label{lem:join-H}
We always have
\begin{eqnarray*}
\cd_2(\HH_1*\HH_2) &= &\min\big(\cd_2(\HH_1)+v(\HH_2),\tau(\HH_1)+\tau(\HH_2),v(\HH_1)+\cd_2(\HH_2)\big).\\
\tau(\HH_1*\HH_2) &= &\min\big(\tau(\HH_1)+v(\HH_2),v(\HH_1)+\tau(\HH_2)\big).
\end{eqnarray*}
\end{lemma}

Let $\K_n^r$ be the complete $r$-uniform hypergraph with $n$ vertices. From an $r$-uniform hypergraph $\HH$, define the hypergraph $\HH(k)$ as the join of $\HH$ with the vertex-disjoint union of $k$ copies of $\K_{2r-1}^r$. We have then the following result.

\begin{proposition}\label{prop:HHk}
If $k\geq \frac{1}{r-1}(v(\HH)-\tau(\HH))$, then $\chi(\KG(\HH(k))=\cd_2(\HH(k))$.
\end{proposition}

\begin{proof}
First, note that 
\begin{equation}\label{eq:cd-tau}
\cd_2(\HH)+v(\HH)\geq 2\tau(\HH).
\end{equation} Indeed, let $X$ be a subset of $V(\HH)$ such that the hypergraph $\HH'$ induced by $V(\HH)\setminus X$ is $2$-colorable; then $X$ together with a color class of a proper $2$-coloring of $\HH'$ provides a vertex cover of $\HH$.

Lemma~\ref{lem:join-H} shows that $\cd_2(\HH(k))=\min(\cd_2(\HH)+(2r-1)k,\tau(\HH)+rk,v(\HH)+k)$. Assume now that $k\geq \frac{1}{r-1}(v(\HH)-\tau(\HH))$. Equation~\eqref{eq:cd-tau} implies then that $\cd_2(\HH(k))=v(\HH)+k$. On the other hand, $\chi(\KG(\HH(k)))\leq v(\HH)+k$: this upper bound is obtained by coloring each edge of $\HH(k)$ containing a vertex of $\HH$ by such a vertex and by using an extra color for each copy of $\K_{2r-1}^r$. Dol'nilov's lower bound gives then the equality.
\end{proof}

\bibliographystyle{plain}
\bibliography{ComplementLineGraph}

\end{document}